\documentclass[12pt]{amsart}

\usepackage{amssymb}
\usepackage{amsmath}
\numberwithin{equation}{section}
\usepackage{breakcites}
\usepackage{color}
\usepackage{graphicx}
\usepackage{epstopdf}
\usepackage{diagbox}
\usepackage{array}

\usepackage{lipsum}

\usepackage{todonotes}
\usepackage{hyperref}

\DeclareMathOperator{\nullsp}{ker}

\DeclareMathOperator{\sgn}{sgn}
\usepackage[vcentermath,enableskew]{youngtab}
\usepackage{ytableau}
\usepackage{amssymb}
\usepackage{rotating}
\usepackage{young}
\usepackage{enumerate}

\renewcommand{\L}{\cal L}
\newcommand{\no}{\noindent}

\newcommand{\be}{\begin{equation}} 
\newcommand{\bea}{\begin{eqnarray}}
\newcommand{\ee}{\end{equation}}
\newcommand{\beas}{\begin{eqnarray*}}
\newcommand{\eea}{\end{eqnarray}}
\newcommand{\eeas}{\end{eqnarray*}}
\newcommand{\non}{\nonumber}
\newcommand{\cal}{\mathcal}
\newcommand{\lie}{\mathrm{Lie}}

\def\S{{\bf S}}

\def\C{{\mathbb C}}

\def\z3{{\mathbb Z_3}}

\def\L{{\cal L}}

\newtheorem{theorem}{Theorem}[section]
\newtheorem{definition}[theorem]{Definition}

\newtheorem{corollary}[theorem]{Corollary}
\newtheorem{conjecture}[theorem]{Conjecture}
\newtheorem{lemma}[theorem]{Lemma}

\newtheorem{proposition}[theorem]{Proposition}
\newtheorem{question}[theorem]{Question}
\begin{document}
\title[CataLAnKe Theorem]{On a generalization of $\lie(k$): a CataLAnKe theorem}
\author[Friedmann]{Tamar Friedmann}
\address{Department of Mathematics and Statistics, Smith College,
and Department of Physics and Astronomy, University of Rochester}
%\email{tamarf1@yahoo.com}
\curraddr{Department of Mathematics and Statistics, Colby College}
\email{tfriedma@colby.edu}

\author[Hanlon]{Phil Hanlon}
\address{Department of Mathematics, Dartmouth College}
\email{philip.j.hanlon@dartmouth.edu}

\author[Stanley]{Richard P. Stanley$^1$}
\address{Department of Mathematics, MIT, and Department of Mathematics, University of Miami}
\email{rstan@math.mit.edu}
\thanks{$^1$Supported in part by NSF grant DMS 1068625}

\author[Wachs]{Michelle L. Wachs$^2$}
\address{Department of Mathematics, University of Miami}
\email{wachs@math.miami.edu}
\thanks{$^{2}$Supported in part by NSF grants
DMS 1202755, DMS 1502606, and by   Simons Foundation grant
\#267236.}

\begin{abstract}We initiate a study of the representation of the symmetric group on the multilinear component of an $n$-ary generalization  of the free Lie algebra, which we call a free LAnKe.  Our central result is that  the representation of the symmetric group $S_{2n-1}$ on the multilinear
  component of the free LAnKe with $2n-1$ generators is given by an irreducible representation
whose dimension is the $n$th Catalan number. This leads to a more general result on eigenspaces of 
  a certain linear operator, which has additional consequences.    We also obtain a new presentation of  Specht modules of staircase shape as a consequence of our central result.  
 \end{abstract} 

%\date{\today}

\maketitle
\section{Introduction} \label{introsec}
%\renewcommand{\baselinestretch}{1.3}
%\normalsize

Lie algebras are defined as vector spaces equipped with an
antisymmetric commutator and a Jacobi identity. They are a cornerstone
of mathematics and have applications in a wide variety of areas of
mathematics as well as physics. Also of fundamental importance is the
free Lie algebra, a natural mathematical construction central in the
field of algebraic combinatorics. The free Lie algebra has beautiful
dimension formulas; an elegant basis in terms of binary trees;
and connections to the shuffle algebra, Lyndon words, necklaces, Witt
vectors, the descent algebra of the symmetric goup, quasisymmetric functions,
noncommutative symmetric functions, and the lattice of set
partitions.  
See \cite{Re} for further information.

In this paper we consider a generalization of the free Lie algebra to $n$-fold commutators, and the representation of the symmetric group on its multilinear component. This representation is a direct generalization of the well-known representation   of the symmetric group  on the multilinear component of the free Lie algebra.  

Let $X:=\{x_1, x_2, \ldots , x_m\}$ be a set of generators. Then the multilinear component  of the free Lie algebra on $X$ is the subspace spanned by bracketed ``words" where each generator in $X$ appears exactly once.  For example,
 $[[x_1, x_3],[[x_4, x_5],x_2]]$ is such a bracketed permutation when $m=5$,
 while $[[x_1, x_3],[[x_1, x_5],x_3]]$ is not.   The symmetric group $S_X$ acts naturally on the bracketed words. Indeed,  $\sigma\in S_X$ acts  by replacing each ``letter"  $x$ of the bracketed word by $\sigma(x)$.  For example,
 $$\sigma [[x_1, x_3],[[x_4, x_5],x_2]] = [[\sigma(x_{1}), \sigma(x_{3})],[[\sigma(x_{4}), \sigma(x_{5})],\sigma(x_2)]].$$
 This induces a representation $\lie(m)$ of the symmetric group $S_m$ on the multilinear component of the free Lie algebra on $m$ generators.  It is well known that the dimension of $\lie(m)$ is $(m-1)!$

The representation $\lie(m)$  has several equivalent descriptions; we
mention two of them here.
\begin{theorem} \label{liek} Let $m \ge 2$.

\begin{enumerate} \item[(a)] {\normalfont(Klyachko \cite{K})} 
The representation $\lie(m)$  is equivalent to the representation
 induced to $S_m$ by any faithful one-dimensional representation of a cyclic subgroup
 of order $m$ generated by an $m$-cycle.
\item[(b)]  {\normalfont (Kraskiewicz and Weyman \cite{KW})}  
Let $i$ and $m$ be relatively prime, and let $\lambda$ be a partition of $m$. The multiplicity of the irreducible representation  indexed by $\lambda$
 in $\lie(m)$ is equal to the number of standard Young tableaux of shape $\lambda$ and of major index congruent to $i \mod m$. 
\end{enumerate}
\end{theorem}

Interestingly, $\lie(m)$ appears in a variety of other contexts, such
as the top homology of the lattice of set partitions in work of
Stanley \cite{St1}, Hanlon \cite{Han}, Barcelo \cite{Ba}, and Wachs
\cite{Wa}, homology of configuration spaces of $m$-tuples of
distinct points in Euclidean space in work of Cohen \cite{Co}, and scattering amplitudes in gauge theories in work of Kol and Shir \cite{KS}.

The generalization of the free Lie algebra that we will consider is based on the following definition.   Throughout this paper,  all vector spaces are taken over the   field $\C$.

\begin{definition} \label{lankedef}  A   {\bf Lie algebra $\cal L$ of the $n$-th 
kind} (a ``LAnKe," or ``LATKe" for $n=3$)   is a vector space equipped with an n-linear  bracket
\[ [\cdot, \cdot , \hskip .7cm , \cdot] : \times ^n \L \rightarrow \L ~\]
that satisfies the following antisymmetry relation for all $\sigma$ in the symmetric group $S_n$:
\begin{equation} \label{antsym} [x_1,\dots,x_n] = \sgn(\sigma) [x_{\sigma(1)},\dots, x_{\sigma(n)}] \end{equation}
and the following generalization of the Jacobi identity:
\begin{align}
\label{type1} &[[x_1, x_2, \ldots,  x_n], x_{n+1},\ldots , x_{2n-1}]
\\ \non&=\sum_{i=1}^{n} [x_1, x_{2}, \ldots , x_{i-1},[ x_{i}, x_{n+1}, \ldots , x_{2n-1}], x_{i+1}, \ldots , x_{n}],
\end{align}
for $x_i \in \L$. 
\end{definition}

The above definition arose  from generalizing a relation between ADE singularities and ADE Lie algebras as a tool to solve a string-theoretic problem (see \cite{Fr}; also, see \cite{Fi, Ta, DT, Ka, Li, BL, Gu}).  This $n$-ary generalization  of Lie algebras is also referred to in the literature as a Filippov algebra.  A different generalization of Lie algebras that also involves $n$-ary brackets appeared in the 1990's in work of Hanlon and Wachs \cite{HW}.

Similar to a free Lie algebra, a LAnKe is free on $X= \{x_1,\dots, x_m\}$ if it is generated
by all possible $n$-bracketings of elements of $X$, and if the only
possible relations existing among these bracketings are consequences
of $n$-linearity of the bracketing, the antisymmetry of the bracketing
(\ref{antsym}), and the generalized Jacobi identity (\ref{type1}).  The multilinear component is spanned by $n$-bracketed permutations of $X$.  Every such $n$-bracketed permutation on $X$ has the same number of brackets, which depends on $m$, the number of generators.  Indeed, if the number of brackets is $k$ then   $m=(n-1)k +1$.

The object we study in this paper is 
the representation of the symmetric group $S_{(n-1)k +1}$ on the
multilinear component of the free LAnKe on $(n-1)k +1$ generators. 
We
denote this representation by $\rho_{n,k}$ and note that $\rho_{2,k}=\lie(k+1)$. See Section~\ref{prelimsec} for further details.

\begin{figure}[!t]

\begin{minipage}{\linewidth}
\small

%%\[
\newcolumntype{L}{>{$}l<{$}} % math-mode version of "l" column type
\newcolumntype{C}{>{$}c<{$}} % math-mode version of "c" column type

\begin{tabular}{|C||C|C|C|C|C||}
\hline 
&\multicolumn{5}{c|}{\mbox{{\bf Table 1: What we can prove about the representations $\rho_{n,k}$}}}\\ 
\hline 
%n$\backslash$ k
${\backslashbox{$n$ \kern-1em}{ \kern-1em $k$}}$
 &1&2&3&4&k \\ \hline 
&S_2&S_3&S_4&S_5&S_{k+1}  \\
&[..]&[.[..]]&[.[.[..]]]&[.[.[.[..]]]]
&\\
2&\yng(1,1)&\yng(2,1)& \yng(2,1,1) \oplus  \yng(3,1) &
\begin{array}{c} 32\oplus  41\\  \oplus  21^3  \oplus  31^2 \oplus  2^21 \end{array} &Lie(k+1)\\ 
&1&2&
6&24&k!\\ \hline
&S_3&S_5&S_7&S_9&S_{2k+1}\\
&[...]&[..[...]]&[..[..[...]]]&[..[..[..[...]]]]&\\
3&\yng(1,1,1)&\yng(2,2,1)&\yng(3,2,1,1)\oplus \yng(3,3,1)&\begin{array}{c}432\oplus  4^21 \\ \oplus  421^3 \oplus  431^2 \oplus 42^21 \\ \oplus  \gamma_{3,4} \footnote{where  $\gamma_{3,4} $ is an $\S_9$-module of dimension $204$; see section \ref{conjecturesec}.}\end{array}  & \rho_{3,k}\\
&1&5&
56&1077 &\\ \hline
&S_4&S_7&S_{10}&S_{13}&S_{3k+1}\\
&[....]&[...[....]]&[...[...[....]]]&[...[...[...[....]]]]&\\
4&\yng(1,1,1,1)&\yng(2,2,2,1)&\yng(3,3,2,1,1)\oplus  \yng(3,3,3,1)& 
\rho_{4,4} & \rho_{4,k}\\
&1&14&
660&& \\ \hline
&S_n&S_{2n-1}&S_{3n-2}&S_{4n-3}&S_{(n-1)k +1}\\ &&&&& \\
n&1^n&2^{n-1}1&3^{n-2}21^2\oplus  3^{n-1}1 &
\rho_{n,4}
& \rho_{n,k}\\ 
&1&{1\over n+ 1}{2n \choose n}&
{4\over \prod_{i=1}^3 (n+ i)} {3n \choose n,n,n}&&\\
&&&&&\\ 
\hline
\end{tabular}

\end{minipage}

\end{figure}

Table 1 summarizes what we can now prove  about the decomposition of $\rho_{n,k}$ into irreducibles. The Young diagrams in the table stand for Specht modules $S^\lambda$ of the indicated shape $\lambda$.   The number given below each decomposition is the dimension of the representation $\rho_{n,k}$, which can be obtained from the well known hook length formula when $k \le 3$. The dimension of $\rho_{3,4}$ was obtained using a C++ computer program.  The symmetric group $S_{(n-1)k +1}$ is given at the top of each cell.
The sign representations that appear in the $k=1$ column  trivially follow from the antisymmetry of the bracket.  The $k=2$ column follows from our central result, Theorem~\ref{catalanke} below, and the  $k=3$ column follows from results of the authors, which will appear in a forthcoming paper \cite{FHSW2}. (This result for $k=3$ was a conjecture in an earlier version of this paper.)  The decomposition for $\rho_{3,4}$ is also proved in \cite{FHSW2}.

\begin{theorem} \label{catalanke} For all $n\ge 2$, the
 $S_{2n-1}$-module $\rho _{n,2}$ is isomorphic to the Specht module $S^{2^{n-1}1}$,
  whose dimension is the $n^{th}$ Catalan number ${1\over
    n+1}{2n\choose n}$.  
\end{theorem}

An explicit $S_{2n-1}$-isomorphism from $\rho_{n,2}$ to $S^{2^{n-1}1}$ can in fact be obtained from  presentations of  free LAnKes and   Specht modules following from results in \cite{DI} and  \cite{Fu}, respectively; see Section~\ref{stair}.   In Section~\ref{catalanrep},  the relationship between $\rho_{n,2}$ and $S^{2^{n-1}1}$ is placed in a  more general setting.  The $S_{2n-1}$-module $\rho_{n,2}$ has a  presentation of the form $V_{n,2}/ R_{n,2}$, where $V_{n,2}$ is generated by $n$-bracketed permutations involving an   $n$-bracket  that is  antisymmetric only, and $R_{n,2}$ is the submodule of $V_{n,2}$ generated by the generalized Jacobi relations (\ref{type1}).  We consider a natural linear operator on $V_{n,2}$ whose kernel  is isomorphic to $\rho_{n,2}$.  We show that  all the  eigenspaces  are  irreducible of the form $S^{2^{i}1^{2n-1-2i}}$ and that the one corresponding to eigenvalue 0 is obtained by setting $i=n-1$.  

 Techniques from our proof also play a role in the proof  of the above mentioned decomposition for $\rho_{n,3}$ obtained in \cite{FHSW2}.  Our eigenspace approach is developed further in subsequent work of Brauner and Friedmann \cite{BF} and of the authors \cite{FHSW3} in which new presentations for Specht modules are obtained; see Section~\ref{stair}.

The first three columns of Table 1 suggest that  
  $\rho_{n,k}$  is isomorphic to the module $\beta_{n,k}$ whose decomposition is  obtained by adding a row of length $k$ to the top of each Young diagram in the decomposition of $\rho_{n-1,k}$.  However the entry $\rho_{3,4}$ shows that this is not always the case since $\gamma_{3,4} \ne 0$.  In  \cite{FHSW2} we  show that $\rho_{n,k}$  contains $\beta_{n,k}$ for all $n,k$, and in Section~\ref{conjecturesec} we speculate on possible necessary and sufficient conditions for  $\rho_{n,k} \cong \beta_{n,k}$.

This paper is organized as follows. 
 In Section~\ref{catalanrep} we prove our  
general result on eigenspaces of the  linear operator on $V_{n,2}$ mentioned above, which yields Theorem~\ref{catalanke}.
In Section~\ref{stair} we discuss   presentations of LAnKes and Specht modules that yield  an explicit isomorphism for Theorem~\ref{catalanke}.  We also use Theorem~\ref{catalanke} to obtain a new presentation of any Specht module of staircase shape.   In Section~\ref{conjecturesec} we discuss further research.

An extended abstract of this work appeared in the proceedings of FPSAC 2018 \cite{FHSW1}.

\section{The CataLAnKe representation and a linear operator} \label{catalanrep}
In this section, we prove Theorem~\ref{catalanke} by placing it   in a more general context as described in the introduction.  We define a linear operator whose kernel is isomorphic to the  representation $\rho_{n,2}$. We show that all the eigenspaces of this operator are irreducible and in particular that the eigenspace corresponding to $0$ is the Specht module $S^{2^{n-1}1}$.

\subsection{Preliminaries} \label{prelimsec}

The following generalizes the standard definition of a free Lie algebra.

\begin{definition} Given a set $X$, a   {\bf free LAnKe on $X$ } is a LAnKe $\L$ together with a mapping $i:X\rightarrow \L$ with the following universal property: for each LAnKe $\cal K$ and each mapping $f:X\rightarrow \cal K$, there is a unique LAnKe homomorphism $F:\L\rightarrow \cal K$ such that $f=F\circ i$.
\end{definition}

From this definition, one can see that the  free LAnKe on $[m]:= \{1,2,\dots,m\}$  is the vector space generated by the elements of $[m]$ and all possible $n$-bracketings involving these 
elements, subject only to the $n$-linearity of the bracket, antisymmetry, and generalized Jacobi relations given in  Definition~\ref{lankedef}.  Let 
$\lie_n(m)$ denote the multilinear component of the free LAnKe on $[m]$, that is, the subspace generated by 
$n$-bracketed words on $[m]$ that contain
each letter of $[m]$ exactly once.   We call these bracketed words, {\it bracketed permutations}.   

For example, the bracketed permutations $[[1,2,3],4,5]$, $[[1,2,4],3,5]$, $[[1,2,5],3,4]$, $[[1,3,4],2,5]$, $[[1,3,5],2,4]$, $[[1,4,5],2,3]$, $[[2,3,4],1,5]$, $[[2,3,5],1,4]$, $[[2,4,5],1,3]$,
$[[3,4,5],1,2]$ span the vector space $\lie_3(5)$.  By the generalized Jacobi relations and the antisymmetry relations, we have
\begin{eqnarray*}[[1,2,3],4,5] &=& [[1,4,5],2,3] + [1,[2,4,5],3] + [1,2,[3,4,5]] \\
&=& [[1,4,5],2,3]  - [[2,4,5],1,3] + [[3,4,5],1,2] . \end{eqnarray*}

 A permutation $\sigma$ in the symmetric group $S_m$ acts naturally on a bracketed permutation  in $\lie_n(m)$ by
replacing each letter $x$ of a bracketed permutation with $\sigma(x)$.  For example, if $\sigma \in S_5$ then
$$\sigma [[2,3,5],1,4] = [[\sigma(2),\sigma(3),\sigma(5)],\sigma(1),\sigma(4)] .$$ Since the antisymmetry and generalized Jacobi relations  are preserved by this action, this induces a representation of  $S_m$ on the vector space $\lie_n(m)$.

Note that if $k$ is the number of brackets of a bracketed permutation in $\lie_n(m)$ then $m = (n-1)k +1$.  (We can also think of the bracketed permutations as rooted plane $n$-ary trees on leaf set $[(n-1)k +1]$; see Section~\ref{conjecturesec}.) 
 Hence $\lie_n((n-1)k+1)$ is  spanned by the bracketed permutations on $[(n-1)k+1]$ with exactly $k$ brackets.   
Let $\rho_{n,k}$ denote the representation of $S_{(n-1)k +1}$ on $\lie_n((n-1)k +1)$.   In this section,  we study $\rho_{n,2}$, the representation of $S_{2n-1}$ on $\lie_n(2n-1)$.

\subsection{A presentation for $\rho_{n,2}$}  
Let $V_{n,2}$ be the multilinear component of the  vector space generated by all possible  $n$-bracketed permutations on $[2n-1]$, subject only to  antisymmetry of the brackets  given in (\ref{antsym}) (but not to generalized Jacobi, (\ref{type1})).  That is, $V_{n,2}$ is the subspace generated by 
$$u_\tau:= [ [\tau_1, \ldots,  \tau_{n}],  \tau_{n+1}, \ldots ,\tau_{2n-1}],$$
where $\tau \in S_{2n-1}$, $\tau_i=\tau(i)$ for each $i$, and $[ \cdot, \dots, \cdot ]$ is the antisymmetric $n$-linear bracket (that does not satisfy the generalized Jacobi relation).

The symmetric group $S_{2n-1}$ acts on generators of $V_{n,2}$ by the following action: for $\sigma,\tau  \in S_{2n-1}$
$$\sigma u_\tau = u_{\sigma\tau}.$$
This induces a representation of $S_{2n-1}$ on $V_{n,2}$ since the action respects the antisymmetry relation.

For each $n$-element subset $S:= \{a_1,\dots,a_n\}$ of $[2n-1]$, let 
$$v_S =  [ [a_1, \ldots , a_{n}], b_1, \ldots ,b_{n-1} ],$$
 where  $ \{b_1 , \cdots , b_{n-1} \}= [2n-1] \setminus S $, and the $a_i$'s and $b_i$'s are in increasing order.   Clearly, 
\begin{equation} \label{basiseq} \left \{v_S : S \in \binom{[2n-1]}{n}\right \}\end{equation} is a basis for $V_{n,2}$.  Thus $V_{n,2}$ has dimension $\binom{2n-1}{n}$.
 
 For each $S \in \binom{[2n-1]}{n}$, use the generalized Jacobi Identity (\ref{type1}), to define   the  relation
\be \label{jacobirel} R_S :=  v_S \, -\sum_{i=1}^n [a_1, \ldots ,a_{i-1}, [a_i, b_1, \ldots , b_{n-1} ], a_{i+1}, \ldots , a_{n}], \hskip 3cm\ee
where $a_1 <\dots < a_n $ and $b_1 < \dots <b_{n-1}$ are as in the previous paragraph.
Let $R_{n,2}$ be the subspace of $V_{n,2}$ generated by the $R_S$.    Then as $S_{2n-1}$-modules   

\begin{equation} \label{Req} V_{n,2}/R_{n,2} \cong \rho_{n,2}.\end{equation}

\subsection{The linear operator $\varphi$.} Now consider the linear operator $\varphi: V_{n,2} \to V_{n,2}$ defined on  basis elements by
$$\varphi(v_S) = R_S.$$  It is not difficult to see that $\varphi$ is an $S_{2n-1}$-module homomorphism whose image is $R_{n,2}$.  
We will need the following lemmas.

 \begin{lemma} \label{decomplem}

\begin{enumerate}[(a)]
\item As  $S_{2n-1}$--modules, 

\[V_{n,2} \cong \bigoplus _{i=0}^{n-1} S^{2^i1^{2n-1-2i}}.\]
\item The operator $\varphi$ acts as a scalar on each irreducible submodule.
\item As  $S_{2n-1}$--modules, \begin{equation} \label{kereq} \nullsp \varphi \cong V_{n,2} / R_{n,2}. \end{equation}
\end{enumerate}
\end{lemma}
\begin{proof} Observe that, due to the antisymmetry of the bracket, the space $V_{n,2}$ constitutes the representation of $S_{2n-1}$ induced from the  sign representation of the Young subgroup $S_n \times S_{n-1}$:

\[ V_{n,2}\cong  (\sgn_n \times \sgn_{n-1} ) \uparrow _{S_n \times S_{n-1}}^{S_{2n-1}}.  \]
Part (a) then follows from Young's rule twisted by the sign representation.  Since Part (a) indicates that $V_{n,2}$ is multiplicity-free, Part (b)  follows from  Schur's lemma. Part (c) follows from Part (b).
\end{proof}

We leave the straightforward proof of the following lemma to the reader. \begin{lemma} \label{coefth} For all $v \in V_{n,2}$, let $\langle v,v_S \rangle$ denote the coefficient of $v_S$ in the expansion of $v$ in the basis given in (\ref{basiseq}).  Then for all $S,T \in \binom{[2n-1]}{n}$,
$$\langle \varphi(v_S),v_T \rangle = \left\{
  \begin{array}{l l}

1& \quad  \text{if } S=T  \\
\\

    (-1)^{d} &  \quad  \text{if } S\cap T = \{d\}\\
\\ 
    0 & \quad \text{if } S\neq T \text{ but } |S\cap T| >1 . \
  \end{array} \right.$$
\end{lemma}

\subsection{The eigenvalues and eigenspaces of $\varphi$} It follows from (\ref{Req}) and (\ref{kereq}) that
\begin{equation} \label{nulleq} \nullsp \varphi \cong \rho_{n,2}.\end{equation} 
Hence   Theorem~\ref{catalanke} says that the kernel of $\varphi$ is isomorphic to the Specht module $S^{2^{n-1}1}$. The next result generalizes this  to all the eigenspaces of $\varphi$. 

\begin{theorem} \label{eigenth}
The operator $\varphi$ has $n$ distinct eigenvalues given by
\be w _i := 1+(n-i)(-1)^{n-i},\ee
for $i=0,1, \ldots , n-1$. Moreover, if $E_i$ is the eigenspace corresponding to $w _i$ then as $S_{2n-1}$-modules,
$$E_i\cong S^{2^i1^{(2n-1)-2i}}$$
for each $i=0,1, \ldots , n-1$. 
\end{theorem}

\begin{proof} By Lemma~\ref{decomplem}, $\varphi$ acts as a scalar on each irreducible submodule. 
 To compute the scalar, we start by letting $t$ be the standard Young tableau of shape $2^i1^{2n-1-2i}$ given by
$$t\, = \, 
\ytableausetup
{mathmode, boxsize=2em}
\begin{ytableau}
\scriptstyle1 &\scriptstyle n+1  \\
\scriptstyle 2 & \scriptstyle n+2  \\
\vdots & \vdots \\
\scriptstyle i &\scriptstyle n+i\\
\scriptstyle i + 1  \\
\vdots \\
\scriptstyle n \\
\scriptstyle n+i+1\\
\vdots \\
\scriptstyle 2n-1
\end{ytableau}
$$

\vskip .4cm
\no Let $C_t$ be the column stabilizer of $t$ and let $R_t$ be the row stabilizer.
Recall that the Young symmetrizer associated with $t$ is defined by $$e_t:=\sum_{\alpha \in R_{t}}  \alpha  \\  \sum_{\beta \in C_t} \sgn(\beta) \beta $$ 
and that the Specht module $S^{2^i1^{2n-1-2i} }$ is the submodule of the regular representation  $\C S_{2n-1}$ spanned by $\{\tau e_t : \tau \in S_{2n-1} \} $.  

 Now set $T:=[n]$, $r_t:= \sum_{\alpha \in R_t} \alpha$ and factor $$\sum_{\beta \in C_t} \sgn(\beta) \beta = f_t d_t,$$ where $d_t$ is the signed sum of permutations in $C_t$ that stabilize $\{1, 2, \ldots , n\}$, $\{ n+1, \ldots , n+i\}$, $\{ n+i+1, \ldots , 2n-1\}$ and $f_t$ is the signed sum of permutations in $C_t$ that maintain the vertical order of these sets. 
So $e_t v_T =  r_t f_t d_t v_T$. Because of the antisymmetry of the bracket, we have

\begin{align} \non d_t v_T &=n!\, ((n+i)-(n+1)+1)!\, ((2n-1)-(n+i+1)+1)!\, \,\, v_T
\\ \non &=n!\,i!(n-i-1)! \,\,\, v_T \, .
\end{align}
Hence $r_t f_t v_T$  is a  scalar multiple of $e_t v_T$ .  Since   the coefficient of $v_T$ in the expansion of $r_t f_t v_T$ is  $1 $,  we have $e_t v_T \ne 0$.

Let $\psi:\C S_{2n-1} \to V_{n,2}$ be the $S_{2n-1}$-module
homomorphism defined  by $\psi(\sigma) = \sigma v_T$, where $\sigma
\in S_{2n-1}$ and $T:=[n]$.   Now consider the  restriction of $\psi$
to the Specht module  $S^{2^i1^{2n-1-2i} }$.    By the irreducibility
of  the Specht module and  the fact that $e_t v_T \ne 0$, this
restriction  is  an  isomorphism from  $S^{2^i1^{2n-1-2i} }$ to the
subspace of $V_{n,2}$ spanned by $\{\tau e_t v_T : \tau \in S_{2n-1}
\} $.  This subspace is therefore the unique subspace of $V_{n,2}$
isomorphic to $S^{2^i1^{(2n-1)-2i}}$.  From here on, we will abuse
notation by letting $S^{2^i1^{(2n-1)-2i}}$ denote the subspace of
$V_{n,2}$ spanned by $\{\tau e_t v_T : \tau \in S_{2n-1} \} $.

 Since $r_t f_t v_T$ is a scalar multiple of $e_tv_T$, it is in $S^ {2^i1^{(2n-1)-2i}}$.  It follows that 
$$\varphi(r_t f_t v_T ) = c r_t f_t v_T ,$$ for some scalar $c$, which we want to show equals $w_i$.
Using the fact that the coefficient of $v_T$ in $r_t f_t v_T$ is 1, we conclude that
$c$  is the coefficient of $v_T$ in $\varphi(r_t f_t v_T)$.  Hence to complete the proof we need only show that
\be \label{aieq} \langle \varphi(r_t f_t v_T), v_T \rangle = w_i := 1+ (n-i) (-1)^{n-i}.\ee

Consider the expansion, 
$$r_t f_t v_T = \sum_{S \in \binom{[2n-1]}{n}} \langle r_t f_t v_T , v_S\rangle v_S,$$
which by linearity yields,
$$\varphi(r_t f_t v_T )= \sum_{S \in \binom{[2n-1]}{n}} \langle r_t f_t v_T , v_S\rangle \varphi(v_S).$$
Hence the coefficient of $v_T$ is given by
$$\langle \varphi(r_t f_t v_T), v_T \rangle = \sum_{S \in \binom{[2n-1]}{n}} \langle r_t f_t v_T , v_S\rangle \langle\varphi(v_S),v_T\rangle.$$
Looking back at Lemma~\ref{coefth}, we see that the   $S=T$ term is $1$, which yields,
$$\langle \varphi(r_t f_t v_T), v_T \rangle = 1+  \sum_{S \in \binom{[2n-1]}{n} \setminus \{T\}} \langle r_t f_t v_T , v_S\rangle \langle\varphi(v_S),v_T\rangle.$$
To get a contribution from an $S \ne T$ term, by Lemma~\ref{coefth}, we must have $S \cap T = \{d\}$ for some $d$, in which case $ \langle\varphi(v_S),v_T\rangle = (-1)^d$.  Hence 
\be \label{dsumeq} \langle \varphi(r_t f_t v_T), v_T \rangle = 1+ \sum_{d=1}^n (-1)^d  \langle r_t f_t v_T , v_{S(d)}\rangle,\ee
where $$S(d) = \{d,n+1,n+2,\dots,2n-1\}.$$

To compute $ \langle r_t f_t v_T , v_{S(d)}\rangle$, we must consider how we get $v_{S(d)}$ from the action of permutations appearing  in $r_tf_t$ on $v_T$.
Recall that $f_t$ is a sum of column permutations $\sigma$ of $t$ (with sign) that maintain the vertical order of $ \{1, 2, \ldots , n\} $, $ \{n+1, \ldots , n+i\} $, and $ \{n+i+1, \ldots , 2n-1\}$. 
In order to get $S(d)$, we have that $\sigma$ fixes $ 1, 2, \ldots , i $ and $n+1, \ldots , n+i$ and then the row permutation $\alpha$ is 
$$\alpha = (1, n+1)(2, n+2)\cdots (i, n+i) (i+1)\cdots (2n-1)$$

\no \underline{and} $\sigma$ interchanges $\{n+i+1, \ldots , 2n-1\}$ with a subset of $\{ i+1, \ldots ,n\}$, leaving one element $d$ of $\{ i+1, \ldots ,n\}$ in rows $i+1, \ldots , n$. 

Since $\sigma$ maintains the vertical order of $1,2,\ldots , n$, it must be that 
$d=i+1$. Thus the summation in (\ref{dsumeq}) is left only with the $d=i+1$ term.
Suppose that $i+1$ goes to row $j$ with $i+1\leq j\leq n$. So

\footnotesize

$$ t\, = \, 
\ytableausetup
{mathmode, boxsize=2em}
\begin{ytableau}
\scriptstyle1 &\scriptstyle n+1  \\
\scriptstyle 2 & \scriptstyle n+2  \\
\vdots & \vdots \\
\scriptstyle i &\scriptstyle n+i\\
\scriptstyle i + 1  \\
\vdots \\
\scriptstyle j-1 \\
\scriptstyle j\\
\scriptstyle j+1\\ 
\vdots \\
\scriptstyle n \\
\scriptstyle n+i+1\\
\vdots \\
\scriptstyle 2n-1
\end{ytableau}
\hskip 1cm \longrightarrow \hskip 1cm
\alpha \sigma t  \, = \, 
\ytableausetup
{mathmode, boxsize=2em}
\begin{ytableau}
\scriptstyle n+1 &\scriptstyle 1  \\
\scriptstyle n+2   & \scriptstyle 2\\
\vdots & \vdots \\
\scriptstyle n+i &\scriptstyle i\\
\scriptstyle n+i+1\\
 \vdots \\
\scriptstyle n+j-1 \\
\scriptstyle i+1 \\
\scriptstyle n+j \\
\vdots\\
\scriptstyle 2n-1\\
\scriptstyle i + 2  \\
\vdots \\
\scriptstyle n
\end{ytableau}
$$
\normalsize
One can easily compute the sign of $\sigma$ by counting inversions or writing $\sigma$ in cycle form as
\beas \sigma &=& (1)\cdots (i)(n+1)\cdots (n+i)(i+1, n+i+1, i+2, n+i+2, \ldots , j)\\
&&(n+j, j+1)(n+j+1, j+2)\cdots (2n-1, n),
\eeas
where the cycle involving $i+1$ is a $((2n-1-2i)-2(n-j))$--cycle. 
So 
\[ \text{sgn}(\sigma)=(-1)^{(2j-2i-1)+1}(-1)^{n-j}=(-1)^{n-j} .\]

In terms of our basis, 
\begin{align*}   \alpha \sigma v_T   =   [ [&{n+1}, n+2, \ldots , {n+i}, {n+i+1}, \ldots, {n+j-1}, {i+1}, {n+j}, \ldots  , {2n-1}], \\ &1, 2, \ldots , i, {i+2}, \ldots , n ].
\end{align*}
To put this basis in canonical form, we need to move ${i+1}$ to the front of the inside bracket, which
yields $ \alpha \sigma v_T = (-1)^{j-1} v_{S(i+1)}$.  
Hence $\sgn(\sigma) \alpha \sigma v_T = (-1)^{n-1} v_{S(i+1)}$.
Since there are $n-i$ positions $j$ where $i+1$ might land, 
$$\langle r_t f_t v_T, v_{S(i+1)} \rangle = (n-i) (-1)^{n-1}.$$
Since  all other terms in the summation in (\ref{dsumeq}) vanish, by
plugging this into (\ref{dsumeq}), we obtain 
(\ref{aieq}), which completes the proof.
\end{proof}

\begin{proof}[Proof of Theorem~\ref{catalanke}]    By Theorem~\ref{eigenth},  since $w _i=0$ for $i=n-1$ only,  $S^{2^{n-1}1}$ is the kernel of $\varphi$.    The result now follows from (\ref{nulleq}).  
\end{proof}

\section{Alternative presentations}\label{stair}
In this section, we discuss  presentations of LAnKes and Specht modules that yield  an explicit isomorphism from $\rho_{n,2}$ to $S^{2^{n-1}1}$.  We also obtain a new presentation for Specht modules of staircase shape.

For each partition $\lambda= (\lambda_1 \geq\dots \geq \lambda_l)$ of $m$,
let $\mathcal T_\lambda$ be the set of Young tableaux of shape
$\lambda$ in which each element of $[m]$ appears once.
Let $M^\lambda$ be the vector space generated by  
$\mathcal T_\lambda$ subject only to column relations, which are of
the form $t+s$, where $s$ is obtained from $t$ by switching two
entries in the same column.  
Given $t \in \mathcal T_\lambda$, let $\bar t$ denote the coset containing $t$
in $M^\lambda$. These cosets, which are called {\it column tabloids}, generate $M^{\lambda}$.
 The symmetric group $S_m$ acts on $\mathcal T_\lambda$ by replacing
 each entry of a tableau by its image under the permutation in $S_m$.
 This induces a representation of $S_m$ on $M^\lambda$.

There are various different presentations of $S^\lambda$ in the literature, which involve the column relations and Garnir relations.  Here we are interested in a presentation of $S^\lambda$ discussed in Fulton \cite[Section 7.4]{Fu}.
The Garnir relations are of the form
$\bar t-\sum \bar s$, where the sum is over all $s\in \mathcal
T_\lambda$ obtained from $t\in \mathcal T_\lambda$ by exchanging any
$k$ entries of a fixed column   with the top $k$ entries of the next
column, while maintaining the vertical order of each of the exchanged
sets.  There is a Garnir relation $g^t_{c,k}$ for every $t \in
\mathcal T_\lambda$, every column $c \in [\lambda_1-1]$,  and  every
$k$ from $1$ to the length of the column $c+1$.  Let $G^\lambda$ be
the subspace of $M^\lambda$ generated by these Garnir
relations. Clearly $G^\lambda$ is invariant under the action of
$S_m$. The  presentation of $S^\lambda$ obtained in Section 7.4 of \cite{Fu} is given by
   \be \label{Fultoneq} M^\lambda/G^\lambda \cong S^\lambda. \ee
   
 On page 102 (after Ex. 15) of \cite{Fu}, a presentation of $S^\lambda$ with fewer relations is given.   The presentation is
\be \label{1fulton} M^{\lambda} / G^{\lambda,1} \cong S^{\lambda} ,\ee where
 $G^{\lambda,1}$ is the subspace of $G^\lambda$ generated by $$\{ g^t_{c,1} : c \in [\lambda_1-1], t \in \mathcal T_\lambda \}  .$$ 
 
 In Appendix 1 of \cite{DI}, a proof that the   generalized Jacobi relations (\ref{type1}) are equivalent to the relations  
\begin{align} \label{filipov} & [[x_1, x_2, \ldots,  x_n], y_{1},\ldots , y_{n-1}] \\
\nonumber = &\sum_{i=1}^{n} [[x_1, x_{2}, \ldots , x_{i-1},y_{1}, x_{i+1}, \ldots , x_{n}],x_{i},y_{2},\ldots , y_{n-1}]
\end{align}
 is given.
 This gives an alternative presentation of $\rho_{n,k}$ for all $n,k$. 

Using the natural correspondence between generators $$[[a_1,\dots,a_n],b_1,\dots,b_{n-1}]$$  of $V_{n,2}$ and column tabloids $\bar t$,  where $t$ is the tableau whose first column is $a_1,\dots,a_n$ and whose second column is $b_1,\dots, b_{n-1}$, we see that the alternative Jacobi relations (\ref{filipov}) correspond to the Garnir relation $g^t_{1,1}$ for $\lambda = 2^{n-1}1$.  Thus  the natural correspondence between generators yields an isomorphism from  $\rho_{n,2}$ to the realization of $S^{2^{n-1}1}$ given in (\ref{1fulton}).

As we have just noted, the presentation (\ref{1fulton}) and the equivalence of the  generalized Jacobi relations (\ref{type1}) and (\ref{filipov})  imply Theorem~\ref{catalanke}.  It is not difficult to see that  conversely Theorem~\ref{catalanke} and the presentation (\ref{1fulton})  imply  the equivalence of the generalized Jacobi relations (\ref{type1}) and (\ref{filipov}) (not just in the free case).  Thus the proof of Theorem~\ref{catalanke} given in Section~\ref{catalanrep} yields a new proof of this equivalence.

The natural  correspondence between generators of $V_{n,2}$ and generators of $M^{2^{n-1}1}$ also takes the generalized Jacobi relations (\ref{type1}) to the Garnir relations $g^t_{1,n-1}$.  This enables us to give another presentation of $S^{2^{n-1}1}$ with fewer relations than that of (\ref{Fultoneq}).  In fact, we can extend this to a wider class of Specht modules.   Suppose the length  of column $c$ of the Young diagram $\lambda$   is $n$ and the length of column $c+1$ is $n-1$.   One of the Garnir relations for  $t \in \mathcal T_\lambda$  is  $g^t_{c,n-1}$, which  is  $ \bar t- \sum \bar s$, where the sum is over all $s$ obtained from $t$ by exchanging the entire column $c+1$ with all but one element of column $c$.  There will be one $s$ for each entry of column $c$ that remains behind in the exchange.  
  
  Suppose column $c$ of $t$ has entries
  $a_1,a_2,\dots,a_n$ reading from top down and column $c+1$ has entries $b_1,\dots, b_{n-1}$, also reading from top down.  
   We can associate $\bar t$ with the bracketed permutation,
  $$[[a_1,a_2,\dots,a_n],b_1,\dots, b_{n-1}],$$ where the bracket is  antisymmetric.  The Garnir relation $g^t_{c,n-1}$ corresponds to the relation
 \begin{align*} [[a_1,& a_2,\dots,a_n],b_1,\dots, b_{n-1}]- \\ & \sum_{i=1}^{n} [[b_1,\dots, b_{i-1}, a_{i}, b_{i}, \dots, b_{n-1}], a_1,\dots,\hat a_{i}, \dots, a_n],\end{align*}
  where $\hat \cdot$ denotes deletion.
  If we move the $a_{i}$ to the front of the inner bracket and move the inner bracket to the place where the $a_{i}$ was deleted, the signs will cancel each other,  and we will get the generalized Jacobi relation  (\ref{type1}).
  It therefore follows from Theorem~\ref{catalanke} that $\{ g^t_{c,n-1}: t \in \mathcal T_\lambda\}$ generates all the other Garnir relations in $\{g^t_{c,k}: t \in \mathcal T_\lambda, k \in [n-1]\}$ for fixed column $c$.  This allows us to reduce the number of relations in the presentation of $S^\lambda$ given in (\ref{Fultoneq}).  We express this in the following result.  
    
  \begin{theorem} \label{garth} Let $(\lambda^\prime_1 \ge \lambda^\prime_2 \ge \dots  \ge \lambda^\prime_j) $ be the conjugate of $
  \lambda\vdash m$.  Let $\tilde G^\lambda$ be the subspace of 
  $M^\lambda$ generated by the union of the sets $\{ g^t_{c,\lambda^\prime_{c
  +1}}: t \in \mathcal T_\lambda\}$  for each column $c$ for which $\lambda^\prime_{c+1} = \lambda^\prime_{c}-1$  and the sets 
  $ \{g^t_{c,k} : t \in \mathcal T_\lambda, k \in [\lambda^\prime_{c+1}] \} $ for the other columns.  Then 
    $$S^\lambda  \cong M^\lambda / \tilde G^\lambda.$$
  \end{theorem}

  We will say that  $\lambda $ is a {\it staircase partition} if
its conjugate  has the form
$(n,n-1,n-2,\dots,n-r)$.   Note that the partition $2^{n-1} 1$ is a staircase partition.  The following result   reduces to Theorem~\ref{catalanke} for the shape $2^{n-1}1$. 

\begin{corollary}  \label{staircase} Let $\lambda$ be a staircase partition of $m$ and let $\tilde G^\lambda$ be the subspace  of $M^\lambda$  generated by $$\{ g^t_{c,\lambda^\prime_{c+1}} : c \in [\lambda_1-1], t \in \mathcal T^*_\lambda \}  ,$$ where $\mathcal T^*_\lambda$ is the set of  Young tableaux of shape $\lambda$ in which each element of $[m]$ appears once and the  columns increase.  Then
$$S^\lambda  \cong M^\lambda / \tilde G^\lambda.$$
\end{corollary}

In \cite{FHSW3} the authors show that Corollary~\ref{staircase} holds for a broader class of partitions, namely the partitions whose conjugate has distinct parts.  The proof is based on a generalization of Theorem~\ref{eigenth}.   In \cite{BF},  Brauner and Friedmann obtain a result analogous to this generalization of Theorem~\ref{eigenth}  and use it to obtain an interesting new presentation of Specht modules of {\it all} shapes, in which the  number of relations has  been similarly reduced. This new presentation implies the presentation (\ref{1fulton}) and is used to give another proof of Theorem~\ref{catalanke}.

\section{Further results and speculations} \label{conjecturesec} 

For $n \ge 2$ and $k \ge 1$, let $\beta_{n,k}$ be the $S_{(n-1)k+1}$-module whose decomposition into irreducibles is obtained by adding a row of length $k$ to the top of each Young diagram in the decomposition of $\rho_{n-1,k}$. (Set $\rho_{1,k} := S^1$.)  For example, using the decomposition of $\rho_{2,3}$ given in Table 1, we have $\beta_{3,3} = S^{321^2} \oplus S^{3^21}$, and using the  decompostion of $\rho_{2,4}$ given in Table 1, we have  
\begin{equation} \label{b34} \beta_{3,4}= S^{432}\oplus S^{4^21}\oplus S^{421^3} \oplus S^{431^2} \oplus S^{4 2^2 1}.\end{equation}

In \cite{FHSW2} we prove that   
\begin{equation} \label{gambet} \rho_{n,k} \cong \gamma_{n,k} \oplus \beta_{n,k}, \end{equation}
for some  $S_{(n-1)k+1}$-module $\gamma_{n,k}$ whose irreducibles  have at most $k-1$ columns.    From Table 1 we  see that
for $1 \le k \le 3$, $\gamma_{n,k} =0$ if and only if $n \ge k$.

\begin{question} \label{conconj1} For general $k \ge 1$, does $n \ge k$ imply $\gamma_{n,k} =0$?
\end{question}

We think that the converse is likely to be true.
\begin{conjecture}\label{conj1} If $n < k$ then $\gamma_{n,k} \ne 0$.
\end{conjecture}

We can see from Table 1 that this conjecture is true whenever $1\le k \le 4$.  It is easy to see that the conjecture is also true when $n=2$. Indeed, $\beta_{2,k} = S^{k,1}$ has dimension $k$ and $\dim \rho_{2,k} = k!$.

We give some further justification for Conjecture~\ref{conj1} by considering a  submodule of $\rho_{n,k}$ spanned  by a certain set of $n$-bracketed permutations.  
It is convenient to think of $n$-bracketed permutations on a finite set $X$ as rooted plane $n$-ary trees on leaf set $X$.  If $T$ is such a tree and $X=\{a\}$, let $[T]$  be the bracketed permutation $a$.  If $|X| > 1$, let $[T]$ be the bracketed permutation  defined recursively by $[[T_1],[T_2],\dots,[T_n]]$, where $T_1,T_2,\dots, T_n$ are the subtrees of the root of $T$ ordered from left to right.  Note that the number of internal nodes of $T$ is equal to the number of brackets of $[T]$.

For example, if $T$ is the ternary tree in which the children of the root are the leaves 1,2,3 ordered from left to right then $[T] = [1,2,3]$.   If $T$ is the ternary tree with subtrees from left to right given by $T_1, T_2,T_3$, where $[T_1] = [1,2,3]$, $[T_2] = [4,5,6]$,
and $[T_3] = [[7,8,9],10,11]$ then $[T] = [[1,2,3],[4,5,6],[[7,8,9],10,11]]$.

We will say that an internal node  of an $n$-ary tree $T $ is {\it abundant} if all of its children are internal nodes.  In the second example given above, the root of $T$ is the only abundant internal node of $T$.   We will say that $T$ is {\it abundant} if it has an abundant internal node.  Thus the $T$ in the first example given above is not abundant, while the $T$ in the second example is.  

Let $\mathcal T_{n,k}$ be the set of  rooted plane $n$-ary trees on leaf set $[k(n-1) +1]$ and let $\alpha_{n,k}$ be the submodule of $\rho_{n,k}$ spanned by $$\{[T] : T \in \mathcal T_{n,k} \mbox{ and  }T \mbox{ is abundant} \}.$$   It follows from Young's rule that all the irreducibles in $\alpha_{n,k}$ have at most $k-1$ columns.  Hence $\alpha_{n,k}$ is isomorphic to a submodule of $\gamma_{n,k}$.  Thus the following conjecture implies Conjecture~\ref{conj1}.

\begin{conjecture} \label{alphaconj1} If $n < k $ then $\alpha_{n,k} \ne 0$.
\end{conjecture}

Note that the only way that this conjecture could be false is if $[T] = 0$ for all abundant $T \in \mathcal T_{n,k}$.  In particular, if it is false 
for $k=n+1$  then the single term relation $$[[x_1,\dots,x_n],[x_{n+1},\dots,x_{2n}], \dots, [x_{(n-1)n +1},\dots, x_{n^2}]]=0 $$ would have to hold for all LAnKe's. This seems unlikely.

The next two propositions respectively show that the converse of Conjecture~\ref{alphaconj1} is true and that the $n=2$ case of the conjecture is true.
\begin{proposition} \label{ampleprop} Let $n \ge 2$ and $k \ge 1$. Then $\mathcal T_{n,k}$ contains an abundant tree if and only if $n < k $.  Consequently, $\alpha_{n,k} = 0$ if $n \ge k$.
\end{proposition}

\begin{proof} Suppose $T \in \mathcal T_{n,k}$ is abundant.  Let $y_1,\dots,y_j$ be the nonabundant internal nodes of $T$ and for each $i$, let $l_i$ be the number of children of   $y_i$ that are leaves.
Clearly $\sum_{i=1}^j l_i = (n-1)k+1$.  We also have $\sum_{i=1}^j l_i  \le jn \le (k-1)n$ since $T$ is abundant.  Hence $(n-1)k+1 \le (k-1)n$, which is equivalent to $n < k $.

Now suppose $n < k $.  We will  construct an  abundant tree in $\mathcal T_{n,k}$.  First let $S$ be any tree in $\mathcal T_{n,k-n-1}$.  So $S$ has leaf set $\{1,2,\dots, m\}$, where $m=(n-1)(k-n-1)+1$.  Replace the leaf $m$ in $S$ with the tree $U$, where $[U]:=$
$$[[m,\dots,m+n-1],[m+n,\dots, m+2n-1],\dots,[m+(n-1)n,\dots, m+n^2-1]]$$ to get the abundant tree $T$.  Since we added $n+1$ internal nodes (or brackets),  $T$ has $k-n-1 + n+1 = k$ internal nodes. Hence $T$ is the desired abundant tree in  $\mathcal T_{n,k}$.
\end{proof}

\begin{proposition} \label{row2prop} For all $k \ge 1$, $\alpha_{2,k} \cong \gamma_{2,k}$.  Consequently, Conjecture~\ref{alphaconj1} is true for $n=2$.
\end{proposition}

\begin{proof}  We need only show that  $$\rho_{2,k} / \alpha_{2,k} \cong S^{k1}$$   since $\beta_{2,k} = S^{k1}$.   This is clearly true when $k< 3$; so assume $k \ge 3$.   It follows from Theorem~\ref{liek} (b) that $S^{k1}$ has multiplicity $1$ in $\rho_{2,k}$.  Since all irreducibles in $\alpha_{2,k}$ have fewer than $k$ columns, $S^{k1}$ is not  in  $\alpha_{2,k}$. Thus $S^{k1}$ is in the quotient  $\rho_{2,k} / \alpha_{2,k}$ with multiplicity 1.  We will show that there are no other irreducibles in the quotient.  

Define the  comb
 $$c_m:=[\dots [[[1,2],3],4],\dots, m] .$$  It is well known that $\{\sigma c_{k+1}: \sigma \in S_{k+1}, \sigma(1) = 1\}$ forms a basis for $\rho_{2,k}$, see e.g. \cite{Wa}.  This is known as the {\it comb basis}.
 
 Let $2< j \le k$. Let $w= [c_{j-1},[j,j+1]]$, $u= [[c_{j-1},j],j+1]$ and $v= [[c_{j-1},j+1],j]$.  By the Jacobi relations, $w=u-v$. It follows that $$[\dots [[w,j+2],j+3],\dots,k+1] $$ $$= [\dots [[u,j+2],j+3],\dots,k+1] - [\dots [[v,j+2],j+3],\dots,k+1] .$$

Since $w$ represents an abundant tree, so does $[\dots [[w,j+2],j+3],\dots,k+1] $.  Hence  $[\dots [[w,j+2],j+3],\dots,k+1]  =0$ in the quotient $\rho_{2,k}/ \alpha_{2,k}$.  This implies that in the quotient $$ [\dots [[u,j+2],j+3],\dots,k+1] = [\dots [[v,j+2],j+3],\dots,k+1]  .$$
But note that $[\dots [[u,j+2],j+3],\dots,k+1] $ is the comb  $c_{k+1}$ and $[\dots [[v,j+2],j+3],\dots,k+1] $ is the comb $(j,j+1)c_{k+1}$.
Hence in the quotient $c_{k+1} = (j,j+1) c_{k+1}$  for  $2<j \le k$.  It follows that $\sigma c_{k+1} = c_{k+1}$ for all $\sigma \in S_{k+1}$ such that $\sigma(1) = 1$ and $\sigma(2) =2$. This implies that there is at most one comb in the quotient whose leftmost leaves are $1,2$.  

The same argument shows that for each $a=2,3,\dots,k+1$, there is at most one comb in the quotient whose leftmost leaves are $1,a$. 
We are thus left with  at most $k$ combs whose leftmost leaf is $1$.  Since the combs whose leftmost leaf is $1$ form a basis for $\rho_{2,k} $, they span $\rho_{2,k}/ \alpha_{2,k}$.  Hence $\rho_{2,k}/ \alpha_{2,k}$ has dimension at most $k$.  But since $S^{k1}$ has dimension $k$ and is contained in the quotient, the quotient must be isomorphic to $S^{k1}$.   
\end{proof}

\begin{question} \label{qestalpgam} Is $\alpha_{n,k} $ isomorphic to $\gamma_{n,k}$ for all $n \ge 2$ and $k \ge 1$?
\end{question}
From Table 1 and Proposition~\ref{ampleprop}, we see that  Question~\ref{qestalpgam} has an affirmative answer whenever $k \le 3$.

\begin{proposition} \label{diag3} If  Question~\ref{qestalpgam} has an affirmative answer for $(n,k) = (3,4)$ then
\begin{equation} \label{diag3eq} \rho_{3,4} \cong  S^{3^2 1^3} \oplus S^{32^3} \oplus \beta_{3,4}.\end{equation}
 If Question~\ref{qestalpgam} has an affirmative answer when $k=4$ then for all $n \ge 3$, $$\rho_{n,4} \cong$$
$$S^{4^{n-3}3^2 1^3} \oplus S^{4^{n-3} 32^3} \oplus S^{4^{n-2}32}\oplus S^{4^{n-1}1}\oplus S^{4^{n-2}21^3} \oplus S^{4^{n-2}31^2} \oplus S^{4^{n-2} 2^2 1} .$$

\end{proposition}

\begin{proof}  From (\ref{gambet}) we have
$$ \rho_{3,4} \cong  \gamma_{3,4} \oplus \beta_{3,4}.$$
 Using a computer program written in C++, we found that $\dim \rho_{3,4}=1077$.   It follows from (\ref{b34}) and the hook length formula that  $\dim \beta_{3,4}= 873$.
Hence
 $\dim \gamma_{3,4} = 204$.  Since we are assuming $\alpha_{3,4} = \gamma_{3,4}$,  we have  $\dim \alpha_{3,4} = 204$.
 
Every abundant tree in $\mathcal T_{3,4}$,  has the form $$[[\sigma(1),\sigma(2), \sigma(3)],[\sigma(4),\sigma(5), \sigma(6)],
[\sigma(7),\sigma(8), \sigma(9)]], $$ where $\sigma \in S_9$.  It follows that $\alpha_{3,4}$ is a submodule of the induction to $S_{9}$ of the wreath product module $\sgn_{3}[\sgn_{3}]$, which  decomposes as\footnote{Sage was used to obtain this.} 
$$S^{1^9} \oplus S^{2^21^5} \oplus S^{2^31^3} \oplus S^{32^3} \oplus S^{3^21^3}.$$  By the  hook length formula, the  respective dimensions of these Specht modules are $1,27,48,84$ and $120$.  Hence $\dim \alpha_{3,4} $ is the sum of some subset of these numbers.   This subset must be $\{84,120\}$ since this is the only subset whose sum is equal to  $\dim \alpha_{3,4}$. It follows that $\alpha_{3,4}= S^{32^3} \oplus S^{3^21^3}$.  Since we are assuming that $\alpha_{3,4} = \gamma_{3,4}$, equation (\ref{diag3eq}) holds.

For $n \ge 4$, the decomposition now follows from (\ref{gambet}), (\ref{diag3eq}), (\ref{b34}), Proposition~\ref{ampleprop}, and the assumption that $\gamma_{n,4} =\alpha_{n,4}$.
\end{proof}

\vskip .5cm
\no \large {\bf Acknowledgements}
\normalsize
\\
We thank Vic Reiner, Jonathan Pakianathan, Thomas McElmurry, Craig Helfgott, Sarah Brauner, Leslie Nordstrom, Marissa Miller, and Harry Stern for helpful discussions.  We also thank the anonymous referee for helpful comments.

\end{document}